\documentclass[a4paper,11pt]{article}

\usepackage{amsfonts}
\usepackage{amssymb}
\usepackage{graphicx}
\usepackage{amsmath}
\usepackage[a4paper]{geometry}
\usepackage{setspace}

\newtheorem{theorem}{Theorem}

\newtheorem{condition}{Condition}

\newtheorem{lemma}{Lemma}

\newtheorem{proposition}{Proposition}

\newtheorem{remark}{Remark}

\newenvironment{proof}[1][Proof]{\noindent\textbf{#1.} }{\ \rule{0.5em}{0.5em}}

\def\lessim{\ \lower4pt\hbox{$\buildrel{\displaystyle <}\over\sim$}\ }

\begin{document}
\onehalfspacing
\title{Spatially Adaptive Density Estimation \\ by Localised Haar Projections}

 \author{Florian Gach, Richard Nickl, and Vladimir Spokoiny
  \\
\\
\textit{University of Cambridge \footnote{Statistical Laboratory, Department of Pure Mathematics and Mathematical Statistics, University of Cambridge, CB30WB, Cambridge, UK. Email: r.nickl@statslab.cam.ac.uk}} ~ and  \textit{Weierstrass Institute Berlin \footnote{Weierstrass Institute for Applied Analysis and Stochastics, Mohrenstrasse 39, 10117 Berlin, Germany. Email: spokoiny@wias-berlin.de. The author is partially supported by Laboratory for Structural Methods of Data Analysis in Predictive Modeling,
MIPT, RF government grant, ag. 11.G34.31.0073. Financial support by the German Research Foundation (DFG) through the Collaborative Research Center 649 ``Economic Risk'' is also gratefully
acknowledged. }}}

\date{First Version: March 2011, This Version: January 2012}

\maketitle

\begin{abstract} Given a random sample from some unknown density $f_0: \mathbb R \to [0, \infty)$ we devise Haar wavelet estimators for $f_0$ with variable resolution levels constructed from localised test procedures (as in Lepski, Mammen, and Spokoiny (1997, Ann.~Statist.)). We show that these estimators satisfy an oracle inequality that adapts to heterogeneous smoothness of $f_0$, simultaneously for every point $x$ in a fixed interval, in sup-norm loss. The thresholding constants involved in the test procedures can be chosen in practice under the idealised assumption that the true density is locally constant in a neighborhood of the point $x$ of estimation, and an information theoretic justification of this practise is given.
\end{abstract}

\section{Introduction}

One of the most enduring challenges in statistical function estimation
is to devise procedures that adapt to the locally variable complexity
of the unknown function. For example, if one observes a random sample
$X_{1},...,X_{n}$ with density $f_{0}:\mathbb{R}\to\mathbb{R}$,
then $f_{0}$ may exhibit \textit{spatially inhomogeneous smoothness}:
The density could be infinitely-differentiable on most of its support
except for a few points $x_{m}$ where it behaves locally like $|x-x_{m}|^{\alpha_{m}}$
for some distinct numbers $\alpha_{m}$. The location of the irregular
points $x_{m}$ will usually not be known, and neither the corresponding
degree of smoothness $\alpha_{m}$. Moreover $f_{0}$ could possess
a so-called \textit{multifractal} behavior, changing its H\"{o}lder exponents
continuously on its domain of definition -- in fact, as shown in Jaffard
\cite{jaff}, `typical' functions in the Besov spaces usually considered
in nonparametric statistics are always multifractal. Donoho and Johnstone
\cite{dj} and Donoho, Johnstone, Kerkyacharian, and Picard \cite{djkp1},
\cite{djkp2} have suggested that methods based on wavelet shrinkage
can, to a certain extent, adapt to spatially inhomogeneous complexity
of the unknown function $f_{0}$. Moreover, Lepski, Mammen, and Spokoiny
\cite{lms} showed that this is not intrinsic to wavelet methods,
and that similar spatial adaptation results can be proved for kernel methods based on locally variable bandwidth choices.

There are several ways in which one can measure spatial adaptivity
of an estimator. A minimal requirement may be to devise a rule $\hat{f}_{n}(x)$
that estimates $f_{0}(x)$ in an optimal way at every point $x$,
and the methods suggested in \cite{dj} and \cite{lms} meet this
requirement. These procedures depend on the point $x$, and the natural
question arises as to how a given procedure performs globally as an
estimator for $f_{0}$. To address this question, Donoho et~al.~\cite{djkp2}
and Lepski et al.~\cite{lms} considered global $L^{r}$-loss, $r<\infty$,
and argued that taking $L^{r}$-loss over Besov-bodies $B(s,p,q)$
where smoothness is measured in $L^{p}$, $r>p$, gives a way to assess
the spatial performance of an estimator. A probably more transparent
approach to the spatial adaptation problem is to consider sup-norm
loss for estimators with locally variable bandwidths: one aims to
find an estimator $\hat{f}_{n}(x)$ that is locally optimal for estimating
$f_{0}(x)$, and simultaneously so for all $x$. This approach was
not considered in the literature so far -- the results \cite{gn0},
\cite{gn}, \cite{gn1}, \cite{golep} address the spatially homogeneous
setting only.

A first contribution of this article is to show that a dyadic histogram
estimator with variable bin size spatially adapts to possibly inhomogeneous
local H\"{o}lder smoothness of $f_{0}$, in global sup-norm loss. More
precisely, for $K(x,y)$ the Haar wavelet projection kernel, we shall
construct \[
\hat{f}_{n}(x)=\frac{2^{\hat{j}_{n}(x)}}{n}\sum_{i=1}^{n}K(2^{\hat{j}_{n}(x)}x,2^{\hat{j}_{n}(x)}X_{i}),\]
 where $\hat{j}_{n}(x)$ is a variable resolution level that depends
both on $x$ and the sample, and show that the random variable \[
\sup_{x}\frac{1}{r(n,x,f_{0})}\left|\hat{f}_{n}(x)-f_{0}(x)\right|\]
 is stochastically bounded, where $r(n,x,f_{0})$ is the optimal risk of an 'oracle estimator' for $f_{0}$ at the point $x$. We show moreover that this rate equals the pointwise minimax rate of adaptive estimation for $f_0(x)$ at every $x$, and that spatial adaptation occurs uniformly in $x$ except near discontinuities of the H\"{o}lder exponent function $t(f,x)$, see after Theorem \ref{holdc} for a detailed discussion.

While this result shows that spatial adaptation is indeed possible
in a strong theoretical way, a drawback shared by most results in
the literature on adaptive estimation remains: The theoretical findings
give no indication whatsoever as to how to choose the numerical constants
in the thresholds that feature in shrinkage- or Lepski-test-based
methods. It has become a common practice that thresholding constants
are chosen according to simulation results where simulations are drawn
as if the true underlying signal is very simple (say, uniform or piecewise
constant). This practise has not had any general theoretical corroboration
until recently Spokoiny and Vial \cite{sv} gave, in a simple Gaussian
regression model, a certain justification based on the idea of `propagation'.
The results in \cite{sv} are heavily tied to the simplicity of the
model used, in particular to the strong Gaussianity assumption employed,
and to the fact that pointwise loss is considered. In the present
paper we show how the ideas of \cite{sv} generalise, subject to some
nontrivial modifications, to nonparametric density estimation. A key
idea in the proofs in \cite{sv}, translated into the density estimation
context, is to replace the sampling distribution by a locally constant
product measure. The 'transportation cost' of this replacement is
easy to control in the Gaussian setting of \cite{sv}, but in the
density estimation case the fluctuations of the likelihood ratios
between the unknown sampling distribution and relevant locally constant
product measures do not obey a Gaussian regime, but turn out to be
of Poisson type, so that the 'Gaussian intuitions' of \cite{sv} could
be entirely misleading. We show however that the main information
theoretic idea of \cite{sv} remains sound in this Poissonian setting
as well: We use a Lepski-type procedure to construct $\hat{j}_{n}(x)$,
and we show that if we compute sharp thresholds for this procedure
as if the true density $f_{0}$ belonged to a family $\mathcal F$
of locally constant densities, then the resulting estimator is spatially adaptive in sup-norm loss. In contrast to the results in \cite{sv},
the rates of convergence we obtain for the risk of the final estimator
are \textit{exact} rate-adaptive.

While the techniques and results of this paper generalise in principle to more complex estimation problems that involve in particular adaptation to higher degrees of smoothness, we prefer to stay within the simpler setting of Haar wavelets, which allows for a clean exposition of the main ideas.

\section{Uniform spatial adaptation using propagation methods}

We will use the symbol $\|g\|_{T}$ to denote the supremum $\sup_{t\in T}|g(t)|$
of a function $g$ over some set $T$, but we will still use the symbol
$\|g\|_{\infty}$ to denote $\sup_{x\in\mathbb{R}}|g(x)|$ if no confusion
can arise.

For any $j\in\mathbb{N}$, we define a dyadic partition of $(0,1]$
into $2^{j}$-many disjoint subintervals by setting $I_{j,k}=(k2^{-j},(k+1)2^{-j}]$,
$k=0,\ldots,2^{j}-1$; and for $0<x\le1$ we denote by $I_{j,k(x)}$
the unique interval containing $x$. For $j\in\mathbb{N}$, $k=1,\ldots,2^{j}-1$,
let $V_{j,k}$ be the space of all bounded density functions on $\mathbb{R}$
that are constant on $I_{j,k}$. Via the local projections\begin{align*}
K_{j,x}(f)(z): & =\begin{cases}
2^{j}\int_{I_{j,k(x)}}f(y)dy & \textrm{if }z\in I_{j,k(x)},\\
f(z) & \textrm{otherwise},\end{cases}\end{align*}
 we map any bounded density $f$ onto $V_{j,k(x)}$. (Note that $K_{j,x}(f)$
is indeed a density since $K_{j,x}(f)$ and $f$ assign the same probability
to the interval $I_{j,k(x)}$.) For $f\in V_{j,k}$ and $j'\ge j$
we clearly have $K_{j',x}(f)=f$.

\subsection{Estimation procedure\label{Section: Estimation-Procedure}}

Let $X,X_{1},...,X_{n}$ be i.i.d.~with bounded density $f_{0}:\mathbb{R}\to[0,\infty)$, $n>1$. We wish to construct a single estimator which estimates $f_{0}(x)$
in an optimal way, uniformly so for points $x$ in the interval $(a,b]$.
We shall take without loss of generality $(a,b]=(0,1]$, and we shall
assume throughout that $f_{0}$ is bounded away from zero on $(0,1]$.
Let $K(x,y)=\sum_{k}\phi(x-k)\phi(y-k)$ be the projection kernel
based on the Haar wavelet $\phi=1_{(0,1]}$. We shall write $K_{j}(x,y)=2^{j}K(2^{j}x,2^{j}y)$,
and the associated linear density estimator is the dyadic histogram
estimator given by\[
f_{n}(j,x):=\frac{1}{n}\sum_{i=1}^{n}K_{j}(x,X_{i}).\]
 We make the important observation that $E_{f}f_{n}(j,x)=2^{j}P_{f}(I_{j,k(x)})$,
which directly follows from the identity $K_{j}(x,y)=2^{j}1_{I_{j,k(x)}}(y)$.
If $f$ is constant on $I_{j,k(x)}$ this in particular implies $E_{f}f_{n}(j,x)=f(x)$.
(In other words: for any locally (at $x$) constant density $f$ the
bias of $f_{n}(j,x)$ equals zero if the resolution level is chosen
fine enough.)

We finally note that the estimator $f_{n}(j,x)$ by construction only
depends on data points falling into $I_{j,k(x)}$. This amounts to
$n2^{-j}$ being the `effective' sample size for estimating $f_{0}$
at $x$.

\subsection{Local choice of the resolution level}

We fix $j_{\max}:=j_{\max,n}\in\mathbb{N}$ satisfying $2^{-j_{\max}} \ge (\log n)^{2}/n$ for some $d > 0$. For thresholds $\zeta_{n}$ to be specified below, and for $J\in\mathbb{N}$, $J\le j_{\max}$ and $0<x\le1$, we define
\begin{eqnarray}
\lefteqn{\hat{j}_{n}(J,x)=\min\biggl\{ j\in\mathbb{N},\, J\le j\le j_{\max}:}\nonumber \\
 &  & \sqrt{n2^{-j'}}\left|f_{n}(j',x)-f_{n}(j,x)\right|\le\zeta_{n} \sqrt{f_{n}(j,x)} \textrm{ for all }j',\, j<j'\le j_{\max}\biggr\}\label{lepskiH}\end{eqnarray}
 as well as \begin{equation}
\hat{j}_{n}(x)=\hat{j}_{n}(0,x).\label{lepski}\end{equation}
 (If the condition in \eqref{lepskiH} is not met for any $j$, $J\le j\le j_{\max}$,
we set $\hat{j}_{n}(J,x)=j_{\max}$.) Given the locally variable resolution
level $\hat{j}_{n}$, we define the family of nonlinear estimators
\begin{equation}
\hat{f}_{n}(J,x):=f_{n}(\hat{j}_{n}(J,x),x),~~~~~\hat{f}_{n}(x):=f_{n}(\hat{j}_{n}(x),x),~~~~~x\in[0,1].\label{estimators}\end{equation}
 These are estimators for $f_{0}(x)$ based on a locally variable
resolution level depending on $x$, and they are density-analogues
of the estimators introduced in \cite{lms} in the context of the
Gaussian white noise model. Note that by construction $\hat{j}_{n}(x)$
is a step function in $x$. Introducing the parameter $J$ will be
useful in what follows -- effectively, $\hat{f}_{n}(J,x)$ is a nonlinear
estimator based on a search over the resolution levels $j\ge J$ that
stops at $j_{\max}$.

\subsection{Threshold choice by propagation}

One of the main challenges for all adaptive procedures is the choice
of the thresholds $\zeta_{n}$ used in the tests defined in \eqref{lepskiH}. Define the standardisation
\[
 \frac{1}{s_n(j,x)} := \begin{cases}
                        \frac{1}{\sqrt{f_n(j,x)}} & \text{if } f_n(j,x) > 0; \\
			0 & \text{otherwise}.
                       \end{cases}
\] We suggest to choose the thresholds in such a way that the following
condition is satisfied:

\begin{condition} \label{prop} Let $\mathcal{F}_{j,k}$ be any triangular
array of subsets of $V_{j,k}$, $j\le j_{\max}$, $k=0,\ldots,2^{j}-1$,
and let $k(m)$ be the unique $k$ such that $I_{j_{\max},m}\subseteq I_{j,k}$.
We say that the thresholds $\zeta_{n}$ satisfy the uniform propagation
condition $\textup{UP}(\alpha,\mathcal{F}_{j,k})$ for some fixed
$\alpha>0$ if for every $n$, every $j\le j_{\max}$, every $m=0,\ldots,2^{j_{\max}}-1$,
and every $f\in\mathcal{F}_{j,k(m)}$ we have that \begin{equation}
E_{f}\left(\sup_{x\in I_{j_{\max},m}}\max_{j\le j'\le j_{\max}}\sqrt{\frac{n2^{-j'}}{\log n}}\left|\frac{\hat{f}_{n}(j',x)-f_{n}(j',x)}{s_n(j',x)}\right|\right)^{2}\le\frac{\alpha}{n2^{2j_{\max}}}.\label{UP}\end{equation}
 \end{condition}
(Note that since $\hat{j}_n(j') \ge j'$ we have that $f_n(j',x) = 0$ implies $\hat{f}_n(j',x) = 0$ for the fully data-driven estimator $\hat{f}_n(j')$, and so the error $|\hat{f}_n(j',x) - f_n(j',x)|$ is then $0$.) An interpretation of this condition can
be given along the following lines: For $0<x\le1$ the class $\mathcal{F}_{j,k(x)}$
contains only densities $f$ that can be exactly reconstructed on
$I_{j,k(x)}$ by $\int K_{j}(x,y)f(y)dy$, so that the bias of the
linear estimator $f_{n}(j',x)$ equals zero locally. In particular, any choice
of the resolution level finer than $j'$ will only increase the variance
without reducing the bias, and we would want $\hat{j}_{n}(j',x)$
to detect that and equal, with large probability, $j'$. This property
of $\hat{j}_{n}$ will then be mirrored in the fact that $\hat{f}_{n}(j',x)-f_{n}(j',x)=0$
for every $j'\ge j$ on an event with large probability, in which
case the l.h.s.~of \eqref{UP} is exactly equal to zero. The quantity
$\alpha/(n2^{2j_{\max}})$ stands for the a priori expected tolerance
for a probabilistic error of $\hat{j}_{n}$ to detect the `correct'
resolution level on each interval $I_{j_{\max},m}$ in this `no-bias'
situation.

The following lemma shows that Condition \ref{prop} is not empty
and that thresholds $\zeta_{n}$ satisfying the uniform propagation
condition exist. It shows furthermore that the thresholds can be taken
to be of order $\sqrt{\log n}$ and independent of $f$, which will
be crucial in understanding the adaptive properties of $\hat{f}_{n}$
below.

\begin{lemma} \label{propver} Let $\mathcal{F}_{j,k}$ equal $V_{j,k}$
intersected with the set\[
\left\{ f:0<\delta\le\inf_{0 < x \le 1}f(x),\,\|f\|_{\infty}\le M\right\} \]
 for some fixed $0<\delta,M<\infty$. Then for every given $\alpha>0$
there exists a numerical constant $\kappa>0$ that depends only on $\alpha$ such
that for any threshold choice \[
\zeta_{n}\ge\kappa\sqrt{\log n}\]
 the uniform propagation condition $\textup{UP}(\alpha,\mathcal{F}_{j,k})$
is at least satisfied for $n$ larger than some index that only depends
on $\delta$ and $M$.
\end{lemma} 

While Lemma~\ref{propver} proves the existence of thresholds of the order $\sqrt {\log n}$ under the uniform propagation condition -- a fact that will be seen to imply adaptivity of $\hat f_n$ below -- it does not suggest a practical choice of $\zeta_n$. Instead, this choice can be made by direct evaluation of \eqref{UP}, as follows: Condition \ref{prop} only concerns the local error bounds
over small intervals $I_{j_{\max},m}$ on which the function $f$
is constant, which effectively means that it suffices to check this
condition only for classes of densities which are constant on the
interval of interest. The particular choice of the interval $I_{j_{\max},m}$
is unimportant. Secondly, all quantities in Condition \ref{prop}
depend on known quantities after $f$ is chosen. By construction of
the estimators $f_{n}$ and $\hat{f}_{n}$ the random variable featuring
in \eqref{UP} -- we call it $T$ -- only depends on the number of
data points falling into each of the (uniquely determined) $j'$-fine
intervals containing $I_{j_{\max},m}$. This observation allows for
an easy computation of the l.h.s.~of \eqref{UP} along the following
lines: Fix $0\le p\le1$. Then, for any $f\in\mathcal{F}_{j,k(m)}$
satisfying $2^{-j}f=p$ on $I_{j,k(m)}$, the number $Z$ of observations
falling into the interval $I_{j,k(m)}$ is binomial $B(n,p)$. Conditionally
on $Z=k$, take $k$-many independent random variables that are uniform
on $I_{j,k(m)}$ and count the number of observations $V_{j'}$ in
each of the $j'$-fine intervals. Then compute $f_{n}$, $\hat{f}_{n}$;
and $T$. This shows that $T$ does only depend on $V_{j'}$, $j\le j'\le j_{\max}$,
and that the l.h.s.~of \eqref{UP} is therefore equal to $E[E[T(V_{j},\ldots,V_{j_{\max}})|Z]]$.

The practical choice of $\zeta_{n}$ can then be obtained via a Monte
Carlo simulation of \eqref{UP} by choosing $\zeta_{n}$ as the smallest
threshold for which \eqref{UP} is satisfied in the simulation for
one specific interval $I_{j_{\max},m}$ uniformly over the class of
all densities constant on this interval. Given $j_{\max}$ and $\alpha$, this procedure has to be performed only for one fixed interval $I_{j_{\max},m},$ and then applies for every $m$ simultaneously.

\subsection{Local small bias condition}

\label{lsbc}

The idea behind Condition \ref{prop} is that we take `idealised'
classes of densities $\mathcal{F}$ for which we compute sharp thresholds
$\zeta_{n}$. The danger arises that the true density $f_{0}$ may
be very different from the elements in $\mathcal{F}$, which may lead
to wrong thresholds (and inference). We have to assess the error that
comes from replacing $f_{0}$ by an element from $\mathcal{F}$, in
a neighborhood of a given point $x$. This can be fundamentally quantified
in terms of the log-likelihood ratio between $f_{0}$ and its local
(at $x$) approximand in $\mathcal{F}$. As we shall see, one of the
deeper reasons behind the fact that propagation methods imply adaptation
results is that this error can be related to the usual bias term in
linear estimation.

\begin{condition} \label{sbc} Given real numbers $\Delta_{j,x}$, $0 < x \le 1$, $j \in \mathbb N \cup \{0\}$ satisfying $\Delta_{l',x}\le\Delta_{l,x}$ for every
$l'>l$, we say that $f_{0}$ satisfies the
local small bias condition at $x \in (0,1]$ and with $\Delta_{j,x}\equiv \Delta_{j,x}(f_{0})$ if\[
\textup{Var}_{K_{j,x}(f_{0})}\log\frac{f_{0}}{K_{j,x}(f_{0})}\le\Delta_{j,x}(f_{0})\]
for all $j\in\mathbb{N}$. \end{condition} 

The local 'cost' of transporting a product measure $\prod_{i=1}^n f_0(x_i)$ to $\prod_{i=1}^n K_{j,x}(f_{0})(x_i)$ can be quantified by $n$ times the variance featuring in the above condition, and we shall have to restrict ourselves to resolution levels $j$ for which this transportation cost is at most a fixed constant times the logarithm of the sample size $n$. The smallest resolution level for which this is still the case will be defined as $j^*(x)$: More precisely, for some fixed positive constant $\Delta$,
define the local resolution level\begin{equation}
j^{*}(x):=j^{*}(x,n,\Delta,f_{0})=\min\left\{ j\in\mathbb{N}:j\le j_{\max},\, n\Delta_{j,x}(f_{0})\le\Delta\log n\right\} .\label{Definition: Local bandwith}\end{equation}
While this is an information-theoretic definition of $j^*$, a key observation of this subsection is that it has the classical `bias-variance'
tradeoff generically built into it for suitable choices of $\Delta_{j,x}(f_{0})$.

\begin{lemma} \label{conc} Suppose $f_{0}$ is bounded by some finite
number $M>0$ and that 
\[
 \inf_{0<x\le1}f_{0}(x)\ge\delta>0.
\]
Then $f_{0}$ satisfies Condition~\ref{sbc} with\[
\Delta_{j,x}(f_{0})=\frac{M}{\delta^{2}}2^{-j}\|f_{0}-K_{j,x}(f_{0})\|_{\infty}^{2}.\]

\end{lemma} 
\begin{proof}
First, observe that $K_{j,x}(f_{0})$ is bounded by $M$ and bounded
below by $\delta>0$. Then, using that $K_{j,x}(f_{0})$ coincides
with $f_{0}$ outside of $I_{j,k(x)}$ and the inequality $|\log x-\log y|\le\max(x^{-1},y^{-1})|x-y|$,
we get\begin{eqnarray*}
 &  & \text{Var}_{K_{j,x}(f_{0})}\log\frac{f_{0}}{K_{j,x}(f_{0})}\\
 &  & \le\int\left(\log\frac{f_{0}(y)}{K_{j,x}(f_{0})(y)}\right)^{2}K_{j,x}(f_{0})(y)dy\\
 &  & \le\int\max\left(f_{0}(y)^{-2},K_{j,x}(f_{0})(y)^{-2}\right)(f_{0}(y)-K_{j,x}(f_{0})(y))^{2}K_{j,x}(f_{0})(y)dy\\
 &  & \le\frac{M}{\delta^{2}}\int(f_{0}(y)-K_{j,x}(f_{0})(y))^{2}dy \le\frac{M}{\delta^{2}}2^{-j}\|K_{j,x}(f_{0})-f_{0}\|_{\infty}^{2}.
\end{eqnarray*}
\end{proof}

\medskip

The lemma shows that the quantity $(n/\log n)\Delta_{j,x}(f_{0})$
can be viewed as the square of the `bias divided by the variance'
of linear projection estimators for $f_{0}(x)$. \emph{Hence, to choose
the smallest $j\le j_{\max}$ such that $(n/\log n)\Delta_{j,x}(f_{0})$
is still bounded by a fixed constant $\Delta$ means to locally balance
the `variance' and `bias' term in the nonparametric setting.}

To be more concrete, let us briefly discuss what this means in the
classical situation where the bias is bounded by local regularity
properties of the unknown density $f_{0}$. Since we are interested
in spatial adaptation, we wish to take locally inhomogeneous smoothness
into account by appealing to local H\"{o}lder conditions: Let $0<t\le1$
and let us say that a function $g:\mathbb{R}\to\mathbb{R}$ is locally
$t$-H\"{o}lder at $x\in\mathbb{R}$ if for some $\eta>0$ \[
\sup_{0<|m|\le\eta}\frac{|g(x+m)-g(x)|}{|m|^{t}}<\infty.\]
 Define further a `local' H\"{o}lder ball of bounded functions \[
\mathcal{C}(t,x,L,\eta):=\left\{ g:\mathbb{R}\to\mathbb{R},\,\max\left(\|g\|_{\infty},\,\sup_{0<|m|\le\eta}\frac{|g(x+m)-g(x)|}{|m|^{t}}\right)\le L\right\} .\]
 Condition~\ref{sbc} then has the following more classical interpretation
in terms of local smoothness properties of $f_{0}$:

\begin{lemma} \label{daub} If $f_{0}\in\mathcal{C}(t,x,L,\eta)$
for some $0<t\le1$, then the local bias $\|f_{0}-K_{j,x}(f_{0})\|_{\infty}$
is bounded by $c2^{-jt}$ for some constant $c\equiv c(t,L,\eta)$. Furthermore,
if 
\[
 \inf_{0<x\le1}f_{0}(x)\ge\delta>0,
\]
then Condition \ref{sbc}
is satisfied with \begin{equation}
\Delta_{j,x}(f_{0})=c^{2}\frac{L}{\delta^{2}}2^{-j(2t+1)}.\label{hbd}\end{equation}
 \end{lemma} 
\begin{proof}
Let $y\in I_{j,k(x)}$ be arbitrary. Then, using the substitution
$2^{j}z=2^{j}y-u$, \begin{eqnarray*}
|f_{0}(y)-K_{j,x}(f_{0})(y)| & = & \left|2^{j}\int_{I_{j,k(x)}}(f_{0}(y)-f_{0}(z))dz\right|\\
 & \le & \int_{-1}^{1}|f_{0}(y)-f_{0}(x)+f_{0}(x)-f_{0}(y-2^{-j}u)|du\\
 & \le & 2|f_{0}(y)-f_{0}(x)|+\int_{-1}^{1}|f_{0}(x)-f_{0}(y-2^{-j}u)|du\end{eqnarray*}
By definition of $x,y,I_{j,k(x)}$ we have $|y-x|\le2^{-j}$, and
also $|y-2^{-j}u-x|\le2^{-j+1}$ by the triangle inequality, so that
for $2^{-j+1}\le\eta$ the last quantity is bounded by $c_{0}2^{-jt}$
in view of $f_{0}\in\mathcal{C}(t,x,L,\eta)$. If $2^{-j}>\eta/2$,
then the quantity in the last display can still be bounded by $6\|f_{0}\|_{\infty}\le6L$,
so that choosing $c_{1}=6L(2/\eta)^{t}$ establishes the desired bound
for $c=\max(c_{0},c_{1})$. To prove the second claim, apply Lemma~\ref{conc}. 
\end{proof}

\medskip

Using the bound from the last lemma to verify Condition \ref{sbc},
we see that, by definition of $j^{*}(x)$ and for $f_{0}\in\mathcal{C}(t,x,L,\eta)$,
\begin{equation}
\sqrt{\frac{n2^{-j^{*}(x)}}{\log n}}\sim\left(\frac{n}{\log n}\right)^{\frac{t}{2t+1}}\label{locad}\end{equation}
 is the locally (at $x$) optimal adaptive rate of convergence, so
that the local small bias condition constructs a minimax optimal resolution
level $j^{*}(x)$ at every $x\in[0,1]$.

\subsection{Main results}

We now state the main results, starting with the following `oracle'
inequality. Note that the oracle $f_{n}(j^{*}(x),x)$ is not an estimator in itself as it depends on unknown quantities.

\begin{theorem} \label{main} Let $\hat{f}_{n}(\cdot)$ be the density estimator defined in \eqref{estimators}
with thresholds $\zeta_{n}$ that satisfy the uniform propagation
condition $\textup{UP}(\alpha,\mathcal{F}_{j,k})$ for some $\mathcal{F}_{j,k}$. Suppose $f_{0}$
satisfies Condition \ref{sbc} for every $0<x\le1$, and let $j^{*}(x)$
be as in \eqref{Definition: Local bandwith}. Then we have \begin{eqnarray}
 &  & E_{f_{0}}\sup_{0<x\le1}\sqrt{\frac{n2^{-j^{*}(x)}}{\log n}}\left|\frac{\hat{f}_{n}(x)-f_{n}(j^{*}(x),x)}{s_{n}(j^{*}(x),x)}\right|\nonumber \\
 &  & \le\frac{\zeta_{n}}{\sqrt{\log n}}+\sqrt{\frac{\alpha}{n}}n^{\Delta e^{4U}}\label{Oracle bound}\end{eqnarray}
 for any $U$ satisfying \begin{equation}
U\ge\sup_{0<x\le1}\left\Vert \log\frac{f_{0}}{K_{j^{*}(x),x}(f_{0})}\right\Vert _{\infty}.\label{U}\end{equation}
\end{theorem}

If $\zeta_{n}=O(\sqrt{\log n})$ -- as follows under the conditions of Lemma \ref{propver}
 -- and if one chooses $\Delta<1/2$, $U$ as in the remark below, then the r.h.s.~of \eqref{Oracle bound} is $O(1)$ as $n$ tends to infinity. Theorem \ref{main} thus implies that the estimator $\hat{f}_{n}$ with resolution
levels chosen by the propagation approach is close to the linear {}`oracle
estimator' evaluated at the locally optimal resolution level $j^{*}(x)$,
and this uniformly so on $(0,1]$.
\begin{remark}

If $\mathcal{F}_{j,k}$ is as in Lemma~\ref{propver} and $f_{0}$
is bounded by $M$ and bounded below by $\delta$, we may apply Lemma~\ref{conc}
\emph{(}using $2^{-j_{\max}}\ge d(\log n)^{2}/n$\emph{)} to obtain the bound\[
\log\frac{f_{0}}{K_{j^{*}(x),x}(f_{0})}=\log\left(1+\frac{f_{0}-K_{j^{*}(x),x}(f_{0})}{K_{j^{*}(x),x}(f_{0})}\right)\le\log\left(1+\frac{\Delta}{dM\log n}\right),\]
which tends to zero as $n$ tends to infinity.
\end{remark}
Our results then imply the following uniform spatial adaptation result:

\begin{theorem} \label{hold} Assume that $f_{0}$ is bounded by $M$ and satisfies $\inf_{0<x\le1}f_{0}(x)\ge\delta>0$. Let $\hat{f}_{n}(\cdot)$ be the density estimator
from \eqref{estimators} with thresholds $\zeta_{n}=O(\sqrt{\log n})$
that satisfy the uniform propagation condition $\textup{UP}(\alpha,\mathcal{F}_{j,k})$ for $\mathcal{F}_{j,k}$ as in Lemma~\ref{propver}. Let $j^{*}(x)$ be as in \eqref{Definition: Local bandwith} with
$\Delta<1/2$ and with $\Delta_{j,x}$ as in Lemma~\ref{conc}. Then
\begin{equation}
\sup_{0<x\le1}\sqrt{\frac{n2^{-j^{*}(x)}}{\log n}}\left|\hat{f}_{n}(x)-f_{0}(x)\right|=O_{{\Pr}_{f_0}}(1).\label{Adaptive rate optimality}
\end{equation}
\end{theorem}

Thus the fully data-driven estimator $\hat f_n$ for $f_0$ achieves the locally optimal risk of the 'oracle' based on $j^*(x)$, uniformly at all points in $(0,1]$. If $j^{*}(x)$ -- with $\Delta<1/2$ -- is based on $\Delta_{j,x}$
as in Lemma~\ref{daub}, then \eqref{Adaptive rate optimality} holds
true and the 'oracle' rate is the adaptive locally minimax rate of convergence
at every $0<x\le1$ where $f_{0}$ is locally $t$-H\"{o}lder with
$0<t\le1$, see the discussion in Section~\ref{lsbc} surrounding \eqref{locad}. This means that at any given point $x$ our estimator is rate-adaptive to local H\"{o}lder smoothness (with the usual $\log n$ penalty for adaptation).

One may ask further if spatial adaptation in the minimax sense occurs \textit{uniformly} for every $x \in (0,1]$. A consequence of Theorem \ref{hold} is the following.

\begin{theorem} \label{holdc} Suppose the assumptions of Theorem \ref{hold} are satisfied and that the true density $f_{0}$ lies
in $\mathcal{C}(t(x),x,L(x),\eta(x))$, $0 < x \le 1$, for some $t(\cdot),L(\cdot),\eta(\cdot)$ that
are bounded and uniformly bounded away from zero on $(0,1]$. Let
$j^{*}(x)$ be as in \eqref{Definition: Local bandwith} with $\Delta<1/2$
and with $\Delta_{j,x}$ as in Lemma~\ref{daub}. Then\[
\sup_{0<x\le1}\left(\frac{n}{\log n}\right)^{t(x)/(2t(x)+1)}\left|\hat{f}_{n}(x)-f_{0}(x)\right|=O_{{\Pr}_{f_0}}(1).\]
 \end{theorem}

The assumptions on the functions $t, L, \eta$ need discussion. For densities that locally look like  $|x-x_m|^{\alpha_m}$ we would wish to choose $t(x)$ equal to their pointwise H\"{o}lder exponents $t(x_m)=\alpha_m$ and $t(x)=1$ otherwise, but then $\eta$ is not uniformly bounded away from zero for points $x \to x_m$. However, Theorem \ref{holdc} holds for \textit{any} choice of the functions $t, L, \eta$ for which $f_0$ satisfies $f_0 \in \mathcal{C}(t(x),x,L(x),\eta(x))$, $0 < x \le 1$. In other words, in the above example we can choose $t(x)=\alpha_m$ on the interval $(x_m-\eta_0, x_m+\eta_0)$ and $t(x)=1$ otherwise, where $\eta_0$ is some arbitrary lower bound for $\eta(x)$. This comes at the expense of not being adaptive near $x_m$, i.e., for $x \in (x_m-\eta_0, x_m+\eta_0) \setminus \{x_m\}$, which is sensible as we cannot expect adaptation for points $x$ arbitrarily close to $x_m$ from a finite sample. Inspection of the proofs (particularly the dependence on $\eta$ in Lemma \ref{daub}) shows that, for fixed $n$, the above theorem holds for densities $\mathcal{C}(t(x),x,L(x), \eta_{(n)})$, $0 < x \le 1$, where $\eta_{(n)}$ can be taken of order $n^{-1/3}$, the binwidth corresponding to the maximal smoothness $t=1$ one wants to adapt to in our setting, and this is again reasonable: H\"{o}lder smoothness of $f_0$ in an interval $[x \pm r_n]$ where $r_n = o(n^{-1/3})$ does not allow to control the bias at $x$ with the locally optimal binwidth of order $n^{-1/3}$. By the same arguments multifractal densities $f_0$ which change their H\"{o}lder exponent continuously can be handled by taking $t(x)$ piecewise constant on a partition of $(0,1]$ into bins of size of order $n^{-1/3}$, the estimator achieving the local uniform minimax rate on each bin of the partition.

\section{Proofs}

\subsection{Proof of Theorem \ref{main}}

A first idea is to use a moment bound, localised at any point $x$
of estimation, on the log-likelihood ratio between $f_{0}$ and its
approximand in $V_{j,k}$. 

\begin{lemma} \label{Lemma: Second moment of the likelihood}

If, for fixed $0<x\le1$,\begin{equation}
\textup{Var}_{K_{j,x}(f_{0})}\log\frac{f_{0}}{K_{j,x}(f_{0})}\le\frac{D\log n}{n}\label{Second moment bound}\end{equation}
 for some $0<D<\infty$ and every $n\in\mathbb{N}$, then, for every
$n\in\mathbb{N}$, \[
E_{K_{j,x}(f_{0})}\left(\prod_{i=1}^{n}\frac{f_{0}(X_{i})}{K_{j,x}(f_{0})(X_{i})}\right)^{2}\le n^{2De^{4U}}\]
 holds for any $U$ satisfying\[
U\ge\left\Vert \log\frac{f_{0}}{K_{j,x}(f_{0})}\right\Vert _{\infty}.\]

\end{lemma} 
\begin{proof}
Since the Kullback-Leibler distance \[
\mathcal{K}(f_{0},K_{j,x}(f_{0}))=-E_{K_{j,x}(f_{0})}\log\frac{f_{0}}{K_{j,x}(f_{0})}\ge0\]
is non-negative, we have\begin{eqnarray*}
E_{K_{j,x}(f_{0})}\left(\prod_{i=1}^{n}\frac{f_{0}(X_{i})}{K_{j,x}(f_{0})(X_{i})}\right)^{2} & \le & \left(E_{K_{j,x}(f_{0})}e^{2\left(\log\frac{f_{0}}{K_{j,x}(f_{0})}-E_{K_{j,x}(f_{0})}\log\frac{f_{0}}{K_{j,x}(f_{0})}\right)}\right)^{n}\end{eqnarray*}
by the i.i.d.~assumption. Using the power series expansion of the exponential function and that the variables in the exponent are centered, one easily
bounds the previous display by\[
\left(1+\frac{2De^{4U}\log n}{n}\right)^{n}\le e^{2De^{4U}\log n}=n^{2De^{4U}}.\]
\end{proof}

\medskip

Here is the proof of Theorem~\ref{main}: We first note that Condition~2
allows us to take $\Delta_{j,x}(f_{0})$ to be constant on the intervals
$I_{j,k}$. Consequently, $j^{*}(\cdot)$ from \eqref{Definition: Local bandwith}
is then constant on every interval $I_{j_{\max},m}$, and we set\[
j_{m}^{*}=\sup_{x\in I_{j_{\max},m}}j^{*}(x).\]
 To prove the theorem, we split \begin{eqnarray*}
 &  & E_{f_{0}}\sup_{0<x\le1}\sqrt{\frac{n2^{-j^{*}(x)}}{\log n}}\left|\frac{\hat{f}_{n}(x)-f_{n}(j^{*}(x),x)}{s_{n}(j^{*}(x),x)}\right|\\
 &  & \le E_{f_{0}}\sup_{0<x\le1}\sqrt{\frac{n2^{-j^{*}(x)}}{\log n}}\left|\frac{\hat{f}_{n}(x)-f_{n}(j^{*}(x),x)}{s_{n}(j^{*}(x),x)}\right|1_{\{\hat{j}_{n}(x)<j^{*}(x)\}}\\
 &  & \quad+E_{f_{0}}\sup_{0<x\le1}\sqrt{\frac{n2^{-j^{*}(x)}}{\log n}}\left|\frac{\hat{f}_{n}(x)-f_{n}(j^{*}(x),x)}{s_{n}(j^{*}(x),x)}\right|1_{\{\hat{j}_{n}(x)\ge j^{*}(x)\}}\\
 &  & =:I+II\end{eqnarray*}
 according to whether $\hat{j}_{n}(x)$ comes to lie below the local
resolution level $j^{*}(x)$ or not. By definition of $\hat j_n(x)$ in \eqref{lepskiH} one immediately has \[
I\le\frac{\zeta_{n}}{\sqrt{\log n}}.\]
 About $II$: Define\begin{equation}
S_{m}=\sup_{x\in I_{j_{\max},m}}\max_{j_{m}^{*}\le j\le j_{\max}}\sqrt{\frac{n2^{-j}}{\log n}}\left|\frac{\hat{f}_{n}(j,x)-f_{n}(j,x)}{s_{n}(j,x)}\right|.\label{S_m}\end{equation}
 Using that on the event $\hat{j}_{n}(x)\ge j^{*}(x)$ we necessarily
have $\hat{f}_{n}(x)=\hat{f}_{n}(j^{*}(x),x)$, we see that\begin{eqnarray}
II & \le & E_{f_{0}}\sup_{0<x\le1}\sqrt{\frac{n2^{-j^{*}(x)}}{\log n}}\left|\frac{\hat{f}_{n}(j^{*}(x),x)-f_{n}(j^{*}(x),x)}{s_{n}(j^{*}(x),x)}\right|\nonumber \\
 & \le & E_{f_{0}}\max_{m}\sup_{x\in I_{j_{\max},m}}\max_{j_{m}^{*}\le j\le j_{\max}}\sqrt{\frac{n2^{-j}}{\log n}}\left|\frac{\hat{f}_{n}(j,x)-f_{n}(j,x)}{s_{n}(j,x)}\right|\nonumber \\
 & \le & 2^{j_{\max}}\max_{m}E_{f_{0}}S_{m}.\label{Bound: II}\end{eqnarray}
 We use the Cauchy-Schwarz inequality to bound\begin{eqnarray*}
 &  & E_{f_{0}}S_{m}\\
 &  & =\int\cdots\int S_{m}(x_{1},\ldots,x_{n})\prod_{i=1}^{n}f_{0}(x_{i})dx_{1}\cdots dx_{n}\\
 &  & =\int\cdots\int S_{m}(x_{1},\ldots,x_{n})\prod_{i=1}^{n}\frac{f_{0}(x_{i})}{K_{j_{m}^{*},x}(f_{0})(x_{i})}\prod_{i=1}^{n}K_{j_{m}^{*},x}(f_{0})(x_{i})dx_{1}\cdots dx_{n}\\
 &  & \le\sqrt{E_{K_{j_{m}^{*},x}(f_{0})}S_{m}^{2}}\,\sqrt{E_{K_{j_{m}^{*},x}(f_{0})}\left(\prod_{i=1}^{n}\frac{f_{0}(X_{i})}{K_{j_{m}^{*},x}(f_{0})(X_{i})}\right)^{2}}\end{eqnarray*}
 by the square-root of the second moment of $S_{m}$ under the `idealised'
density $K_{j_{m}^{*},x}(f_{0})$ times the square-root of the second
moment of the likelihood ratio. (Here, $x$ is any point in $I_{j_{\max},m}$.)
Using Condition~1 and Lemma~\ref{Lemma: Second moment of the likelihood},
we obtain a bound for the last term in \eqref{Bound: II} of order\[
2^{j_{\max}}\max_{m}E_{f_{0}}S_{m}\le\sqrt{\frac{\alpha}{n}}n^{\Delta e^{4U}},\]
which concludes the proof of the theorem.

\subsection{Proof of Theorems \ref{hold} and \ref{holdc}}

We first prove Theorem \ref{hold}: Clearly, 
\begin{align*}
   \sup_{0<x\le1}\sqrt{\frac{n2^{-j^{*}(x)}}{\log n}}\left|\hat{f}_{n}(x)-f_{0}(x)\right| &\le\sup_{0<x\le1}\sqrt{\frac{n2^{-j^{*}(x)}}{\log n}}\left|\frac{\hat{f}_{n}(x)-f_{n}(j^{*}(x),x)}{s_{n}(j^{*}(x),x)}\right|\sqrt{f_{n}(j^{*}(x),x)} \\
 & ~~~~~+\sup_{0<x\le1}\sqrt{\frac{n2^{-j^{*}(x)}}{\log n}}\left|f_{n}(j^{*}(x),x)-f_{0}(x)\right|.
 \end{align*}
 The first factor of the first summand is bounded in probability in
view of Theorem \ref{main} and of Lemma \ref{propver} and the hypothesis $\zeta_n = O(\sqrt{\log n})$. The second
factor of the first summand is also bounded in probability since \[
\sup_{0<x\le1}\max_{j\le j_{\max}}|f_{n}(j,x)-E_{f_{0}}f_{n}(j,x)|=o_{P_{f_{0}}}(1)\]
 by Proposition \ref{uexp}, using $2^{-j_{\max}}\ge d(\log n)^{2}/n$,
and since $\sup_{x,j}|E_{f_{0}}f_{n}(j,x)|\le\|f_{0}\|_{\infty}<\infty$.
It remains to prove that the second summand is bounded in probability,
and we achieve this by bounding the moment \begin{eqnarray*}
 &  & E_{f_{0}}\sup_{0<x\le1}\sqrt{\frac{n2^{-j^{*}(x)}}{\log n}}\left|f_{n}(j^{*}(x),x)-E_{f_{0}}f_{n}(j^{*}(x),x)\right|\\
 &  & \qquad+\sup_{0<x\le1}\sqrt{\frac{n2^{-j^{*}(x)}}{\log n}}\left|E_{f_{0}}f_{n}(j^{*}(x),x)-f_{0}(x)\right|\\
 &  & \le E_{f_{0}}\sup_{0<x\le1}\max_{j\le j_{\max}}\sqrt{\frac{n2^{-j}}{\log n}}\left|f_{n}(j,x)-E_{f_{0}}f_{n}(j,x)\right|\\
 &  & \qquad+\max_{m}\sup_{x\in I_{j_{\max},m}}\sqrt{\frac{n2^{-j^{*}(x)}}{\log n}}|E_{f_{0}}f_{n}(j^{*}(x),x)-f_{0}(x)|.\end{eqnarray*}
 The first term is bounded by a fixed constant using Proposition \ref{uexp}
below. Recalling the definition of $j_{m}^{*}$ from the beginning
of the proof of Theorem~\ref{main} and choosing $\Delta_{j,x}(f_{0})$
from Lemma~\ref{conc}, the second term is bounded by\begin{eqnarray*}
 &  & \max_{m}\sqrt{\frac{n2^{-j_{m}^{*}}}{\log n}}\sup_{x\in I_{j_{\max},m}}|E_{f_{0}}f_{n}(j_{m}^{*},x)-f_{0}(x)|\\
 &  & \le\max_{m}\sqrt{\frac{n2^{-j_{m}^{*}}}{\log n}}\|K_{j_{m}^{*},x}(f_{0})-f_{0}\|_{\infty}\\
 &  & \le\delta\sqrt{\frac{\Delta}{M}},\end{eqnarray*}
where $x$ is any point in $I_{j_{\max},m}$, and this completes the proof.

We next prove Theorem \ref{holdc}: Using the hypotheses on $t(\cdot),L(\cdot),\eta(\cdot)$,
the proof of Lemma~\ref{daub} shows that $f_{0}$ satisfies Condition~\ref{sbc}
with\begin{equation}
\Delta_{j,x}(f_0)=c'2^{-j(2t(x)+1)},\label{Delta independent of x}\end{equation}
$0<c'<\infty$, where $c'$ does not depend on $x$. Using that $t(\cdot)$
is bounded below by some positive number implies that 
$$\Delta_{j_{\max},x}(f_0)=c' 2^{-j_{\max}(2t(x)+1)} \le \frac{\Delta \log n}{n}.$$
holds for $n$ large enough (independent of $x \in (0,1]$), so that $j^{*}(x)$, when based on $\Delta_{j,x}(f_0)$ as in $\eqref{Delta independent of x}$, is asymptotically equivalent
to the minimax optimal locally adaptive rate, uniformly so for all
$x$.

\subsection{Proof of Lemma \ref{propver}}

The proof relies on Propositions \ref{uexp1} and \ref{uexp} which
are given below. Recall first from Section~\ref{Section: Estimation-Procedure}
that for $f\in V_{j,k}$ and $j'\ge j$ we necessarily have $E_{f}f_{n}(j',x)-f(x)=0$
for every $x\in I_{j,k}$, so the bias at $x\in I_{j,k}$ is exactly
zero, a fact we shall use repeatedly below without separate mentioning.
Write \begin{eqnarray}
 &  & \sup_{x\in I_{j_{\max},m}}\max_{j' \ge j}\sqrt{\frac{n2^{-j'}}{\log n}}\left|\frac{\hat{f}_{n}(j',x)-f_{n}(j',x)}{s_n(j',x)}\right|\notag\label{oldest}\\
 &  & =\sup_{x\in I_{j_{\max},m}}\max_{j' \ge j}\sqrt{\frac{n2^{-j'}}{\log n}}\sum_{l > j'}\left|\frac{f_{n}(l,x)-f_{n}(j',x)}{s_n(j',x)}\right|1_{\{\hat{j}_{n}(j',x)=l\}}.\end{eqnarray}
 To treat the indicator, observe that\begin{eqnarray*}
 &  & \{\hat{j}_{n}(j',x)=l\} \subseteq\left\{ \sqrt{n2^{-l'}}|f_{n}(l',x)-f_{n}(l-1,x)|>\zeta_{n} \sqrt{f_{n}(l-1,x)} \textrm{ for some }l' \ge l\right\} \\
& & \subseteq \Biggl\{ \sqrt{n2^{-l'}} |f_{n}(l',x) - E_f f_n(l',x) + E_f f_n(l-1,x) - f_{n}(l-1,x)| \\
& & \quad > \zeta_n \frac{\sqrt{f(x)}}{2} \textrm{ for some }l' \ge l \Biggr\}  \cup \left\lbrace \min_{\ell \ge j} \sqrt{f_n(\ell,x)} \le \frac{\sqrt{f(x)}}{2}\right\rbrace
\end{eqnarray*}
 Observe that the first set is a subset of\begin{eqnarray*}
 &  & \left\{ \sqrt{n2^{-l'}}\left| f_{n}(l',x)-E_{f}f_{n}(l',x)\right| \ge \frac{\zeta_{n}\sqrt{f(x)}}{4} \textrm{ for some } l' \ge l\right\} \\
 &  & \quad\cup\left\{ \sqrt{n2^{-(l-1)}}\left| f_n(l-1,x)-E_f f_{n}(l-1,x)\right| > \frac{\zeta_{n}\sqrt{f(x)}}{4} \right\} \\
 &  & \subseteq\left\{ \max_{\ell \ge j} \sqrt{n2^{-\ell}}\|f_{n}(\ell)-E_{f}f_{n}(\ell)\|_{I_{j_{\max},m}} > \frac{\zeta_{n}\sqrt{\|f\|_{I_{j,k(m)}}}}{4} \right\} =:B_{1}
\end{eqnarray*}
and that, using $y \ge \sqrt{\delta y}$ for $y \ge \delta$, the second set is contained in
\begin{eqnarray*}
 & & \left\lbrace \max_{\ell \ge j} |f_n(\ell,x) - E_f f_n(\ell,x)| > \frac{f(x)}{2} \right\rbrace  \\
& & \subseteq \left\lbrace \max_{\ell \ge j} |f_n(\ell,x) - E_f f_n(\ell,x)| > \frac{\sqrt{\delta f(x)}}{2} \right\rbrace \\
& & \subseteq \left\lbrace  \sup_{x \in I_{j_{\max},m},\, \ell \ge j} |f_n(\ell,x) - E_f f_n(\ell,x)| > \frac{\sqrt{\delta \|f\|_{I_{j,k(m)}}}}{2}\right\rbrace := B_2;
\end{eqnarray*}
so that $\{\hat{j}_{n}(j',x)=l\}\subseteq B_{1}\cup B_{2}=:B$, a
set which does not depend on $j',x$ or $l$. Hence, $1_{\{\hat{j}_{n}(j',x)=l\}}\le1_{B}$
uniformly in $j',x,l$, so that the quantity in \eqref{oldest} is
bounded from above by \[
1_{B}\sup_{x\in I_{j_{\max},m}}\max_{j' \ge j}\sqrt{\frac{n2^{-j'}}{\log n}}\sum_{l > j'}\left|\frac{f_{n}(l,x)-f_{n}(j',x)}{s_n(j',x)}\right|,\]
and therefore the second moment of \eqref{oldest} is bounded, using
the Cauchy-Schwarz inequality, by \[
{\Pr}_f(B)^{1/2}\left\Vert \sup_{x\in I_{j_{\max},m}}\max_{j' \ge j}\sqrt{\frac{n2^{-j'}}{\log n}}\sum_{l > j'}\left|\frac{f_{n}(l,x)-f_{n}(j',x)}{s_n(j',x)}\right|\right\Vert _{4,{\Pr}_f}^{2}=:I\times II.\]
 We first bound $II$: By the triangle inequality and since the bias
is exactly zero, this term is less than or equal to \begin{eqnarray}
 &  & 2\left\Vert \sup_{x\in I_{j_{\max},m}}\max_{j' \ge j}\sqrt{\frac{n2^{-j'}}{\log n}}\sum_{l > j'} \left|\frac{f_{n}(l,x)-E_{f}f_{n}(l,x)}{s_n(j',x)}\right|\right\Vert _{4,{\Pr}_f}^{2}\notag\label{nuis}\\
 &  & \quad + 2\left\Vert \sup_{x\in I_{j_{\max},m}}\max_{j' \ge j}\sqrt{\frac{n2^{-j'}}{\log n}}\sum_{l > j'}\left|\frac{f_{n}(j',x)-E_{f}f_{n}(j',x)}{s_n(j',x)}\right|\right\Vert _{4,{\Pr}_f}^{2}.
\end{eqnarray}
Define now $S=\{\sup_{x\in I_{j_{\max},m}}\min_{j' \ge j}f_{n}(j',x)\ge\delta/2\}$.
Note that, by definition of $f_{n}(j')$, $f_{n}(j',x)>0$ implies
$f_{n}(j',x)\ge2^{j'}/n$. Then, for every $1\le p<\infty$, \begin{eqnarray*}
 &  & E_{f}\left(\sup_{x\in I_{j_{\max},m}}\max_{j' \ge j}\frac{1}{s_{n}(j',x)}\right)^{p}=E_{f}\left(\sup_{x\in I_{j_{\max},m}}\max_{j' \ge j}\frac{1}{s_{n}(j',x)}(1_{S}+1_{S^{c}})\right)^{p}\\
 &  & \le\frac{2^{3p/2-1}}{\delta^{p/2}}\\
 &  & \quad+2^{p-1}n^{p/2}E_{f}1\left\{ \sup_{x\in I_{j_{\max},m}}\min_{j' \ge j}|f_{n}(j',x)-E_{f}f_{n}(j',x)+f(x)|<\frac{\delta}{2}\right\} \\
 &  & \le\frac{2^{3p/2-1}}{\delta^{p/2}}+2^{p-1}n^{p/2}{\Pr}_f\left\{ \sup_{x\in I_{j_{\max},m},\, j' \ge j}|f_{n}(j',x)-E_{f}f_{n}(j',x)|>\frac{\delta}{2}\right\} \\
 &  & \le\frac{2^{3p/2-1}}{\delta^{p/2}}\\
 &  & \quad+2^{p-1}n^{p/2}{\Pr}_f\left\{ \sup_{x\in I_{j_{\max},m},\, j' \ge j}\sqrt{n2^{-j'}}|f_{n}(j',x)-E_{f}f_{n}(j',x)|>\frac{\delta\sqrt{d}\log n}{2}\right\} \\
 &  & \le\frac{2^{3p/2-1}}{\delta^{p/2}}+2^{p-1}n^{p/2}cn^{-\frac{\delta^{2}d}{4c}\log n}\end{eqnarray*} for large $n$ in view of Proposition \ref{uexp1} (using that $2^{-j_{\max}}\ge d(\log n)^{2}/n$),
so that this expectation is bounded uniformly in $n$ by some constant
$c_{1}(p,\delta,M)$.
Using this, the Cauchy-Schwarz inequality and
Proposition \ref{uexp}, the square of the first term in \eqref{nuis}
is less than or equal to \begin{eqnarray*}
& & c_2 2^{2j_{\max}} j_{\max}^4 \\
& & \quad \times E_f \left( \sup_{x\in I_{j_{\max},m}}\max_{l \ge j}\sqrt{\frac{n2^{-l}}{\log n}} \left|f_{n}(l,x)-E_{f}f_{n}(l,x)\right| \sup_{x\in I_{j_{\max},m}}\max_{l \ge j}\frac{1}{s_{n}(j',x)} \right) ^4 \\
 &  & \le c_{2} 2^{2j_{\max}} j_{\max}^4 \left(E_{f}\left(\sup_{x\in I_{j_{\max},m}}\max_{l \ge j}\sqrt{\frac{n2^{-l}}{\log n}|}f_{n}(l,x)-E_{f}f_{n}(l,x)|\right)^{8}\right)^{1/2}\\
 &  & \quad\times\left(E_{f}\left(\sup_{x\in I_{j_{\max},m}}\max_{l \ge j}\frac{1}{s_{n}(l,x)}\right)^{8}\right)^{1/2}\le c_{3} 2^{2j_{\max}} j_{\max}^4;\end{eqnarray*}
 and the same reasoning also implies that the second term in \eqref{nuis}
is less than or equal to some constant, so that we can conclude, using
the lower bound of $2^{-j_{\max}}$, that \begin{equation}
II\le c_{4}n\label{two}\end{equation}
 for some fixed constant $c_{4}$ that depends only on $\delta$ and
$M$.

To bound $I$, we have the following: First, using Proposition \ref{uexp1}
below, we see \begin{eqnarray}
{\Pr}_f(B_{1}) & = & {\Pr}_f \left\{ \max_{\ell \ge j}\sqrt{n2^{-\ell}}\|f_{n}(\ell)-E_{f}f_{n}(\ell)\|_{I_{j_{\max},m}}>\frac{\zeta_{n}\sqrt{\|f\|_{I_{j,k(m)}}}}{4}\right\} \nonumber \\
& \le & Dn^{-\frac{\kappa^{2}\delta}{4D}} \label{infcon}\end{eqnarray}
for large $n$, with $D$ only depending on $M$. Furthermore, using $2^{-j_{\max}}\ge d(\log n)^{2}/n$
and Proposition~\ref{uexp1} below,\begin{eqnarray*}
& & {\Pr}_f(B_2) \\
 &  & \le {\Pr}_f\left\{ \sup_{x\in I_{j_{\max},m},\, \ell \ge j}\sqrt{n2^{-\ell}}|f_{n}(\ell,x)-E_{f}f_{n}(\ell,x)|>\frac{\sqrt{d\delta\|f\|_{I_{j,k(m)}}}\log n}{2}\right\} \\
 &  & \le Dn^{-\frac{d\delta^{2}}{4D}\log n}\end{eqnarray*}
for large $n$. Thus, choosing $\kappa$ large enough but finite depending on the
choice of $\alpha$, we obtain for $n$ large enough \[
I\times II\le c_{4}Dn\left(n^{-\frac{\kappa^{2}\delta}{4D}}+n^{-\frac{d\delta^{2}}{4D}\log n}\right)\le\frac{\alpha}{n2^{2j_{\max}}}.\]
 This completes the proof.

\subsection{Uniform-in-bandwidth bounds for Haar wavelet density estimators and
some consequences}

\label{ube}

The following exponential inequality was used repeatedly in the proofs.

\begin{proposition} \label{uexp1}Let $j_{\max}\in\mathbb{N}$ such
that $2^{-j_{\max}} \ge d(\log n)^2/n$. Let $I=(2^{-j}k,2^{-j}(k + 1)]$ for some $j \le j_{\max}$ and $k\in \mathbb{Z}$, and suppose $f:\mathbb{R} \rightarrow [0,\infty)$ is a density that satisfies $\|f\|_I \le M$ and
\[
 \inf_{x \in I} f(x) \ge \delta > 0.
\]
There exist constants $C_1(d)$, $C_2(d)$ and an index $n(\delta,M)$ such that for all $n \ge n(\delta,M)$ and all $C_3 \ge C_2(d)$, if
\begin{equation} \label{range}
C_{1}(d)\sqrt{\|f\|_{I} \log n} \le u \le C_{3} \|f\|_{I} \sqrt{n2^{-j_{\max}}},
\end{equation}
then 
\[
{\Pr}_f \left\{ \sup_{x\in I}\max_{j\le j' \le j_{\max}}\sqrt{n2^{-j'}}|f_{n}(j',x)-E_f f_{n}(j',x)|\ge u\right\} \le De^{-\frac{u^{2}}{D}},
\]
where $D$ only depends on $C_3$ and $M$.
\end{proposition} 
\begin{proof}
Writing \begin{eqnarray*}
& &  \sqrt{n2^{-j'}}\left|f_{n}(j',x)-E_{f}f_{n}(j',x)\right| \\
& & = 2 \sqrt{\frac{2^{j_{\max}}}{n}} \frac{\sqrt{2^{j'-j_{\max}}}}{2}\left|\sum_{i=1}^{n}(K(2^{j'}x,2^{j'}X_{i})-E_f K(2^{j'}x,2^{j'}X_i))\right|,
\end{eqnarray*}
 we have to consider the supremum \[
2 \sqrt{\frac{2^{j_{\max}}}{n}}\sup_{h\in\mathcal{H}}\left|\sum_{i=1}^{n}(h(X_{i})-E_f h(X_i))\right|\]
 of the (scaled) empirical processes indexed by the class of functions
\[
\mathcal{H}:=\left\{ \frac{\sqrt{2^{j'-j_{\max}}}}{2}K(2^{j'}x,2^{j'}(\cdot)):x\in I,\, j' \ge j \right\} .\]
 This class has constant envelope $1/2$ since $j' \le j_{\max}$ and
since $\sup_{x,y}|K(x,y)|=1$. Furthermore, noting that $K^{2}(x,y)=K(x,y)$
for every $x,y$, we have for $h\in\mathcal{H}$ that
\begin{eqnarray*}
E_f h^{2}(X) & = & \frac{2^{j'-j_{\max}}}{4}\int K^{2}(2^{j'}x,2^{j'}y)f(y)dy \\
& = & \frac{2^{j'-j_{\max}}}{4} \int_{2^{-j'}k(x)}^{2^{-j'}(k(x) + 1)} f(y) dy \le\frac{2^{-j_{\max}}}{4} \|f\|_{I}.
\end{eqnarray*}
 Note further that $\mathcal{H}$ is a VC-type class of functions
by using Lemma 2 in \cite{gn} and a simple computation on covering
numbers (including an obvious covering of the set $[2^{-j_{\max}},1]\subseteq[0,1]$).
Rewrite\begin{eqnarray*}
 &  & {\Pr}_f \left\{ \sup_{x\in I}\max_{j\le j' \le j_{\max}}\sqrt{n2^{-j'}}|f_{n}(j',x)-E_f f_{n}(j',x)|\ge u\right\} \\
 &  & = {\Pr}_f \left\{ \sup_{h\in\mathcal{H}}\left|\sum_{i=1}^{n}(h(X_{i})-E_f h(X_i))\right|\ge \frac{u\sqrt{n2^{-j_{\max}}}}{2}\right\} \end{eqnarray*}
 and apply expression (21) in \cite{gn1}, with
\[
 \sigma^2 := \frac{2^{-j_{\max}}\|f\|_I}{4} \wedge \frac{1}{4}
\]
and
\[
 \lambda := \begin{cases}
            c_1(d) \sqrt{\frac{\log n}{\|f\|_I \log \frac{n}{\|f\|_I}}} & \text{if } \|f\|_I \le 1, \\
		c_2(d) & \text{otherwise};
           \end{cases}
\]
for appropriate constants $c_1(d)$, $c_2(d)$ that only depend on $d$.
\end{proof}
\begin{proposition} \label{uexp} Let $j_{\max}$, $I$ and $f$ be as in Proposition~\ref{uexp1}. Then there exists a
constant $D(d,\delta,M)$ such that
for every $1\le p<\infty$ we have \begin{equation}
E_f \left(\sup_{x \in I}\max_{j\le j' \le j_{\max}}\sqrt{\frac{n2^{-j'}}{\log n}}\left|f_{n}(j',x)-E_f f_{n}(j',x)\right|\right)^{p}\le D^{p}.\end{equation}
 \end{proposition} 
\begin{proof}
The proof follows from considering the same empirical process as in
the proof of Proposition \ref{uexp1}, and using bounds for $p$-th
moments of empirical processes indexed by uniformly bounded VC-classes
of functions, e.g., the bound in the display following (21) in \cite{gn1}, with $\sigma^2$ and $\lambda$ as in the proof of Proposition \ref{uexp1}, together with Proposition~3.1 in \cite{glz}. \end{proof}

\medskip

\noindent \textbf{Acknowledgement.} The authors would like to thank an anonymous referee for critical remarks, particularly on Theorem \ref{holdc}.

\end{document}